\documentclass[a4paper,12pt]{amsart}
\usepackage{amssymb,amsmath,amsfonts, amsthm}
\usepackage{hyperref}
\usepackage{dsfont}
\usepackage{upgreek}
\usepackage{mathrsfs}
\usepackage{mathabx,yfonts}

\newtheorem{thm}{Theorem}

\newtheorem{lem}{Lemma}

\newtheorem*{ack}{Acknowledgements}

\numberwithin{equation}{section}

\begin{document}

\title{The saturation number for Cayley's cubic surface}

\author{Yuchao Wang}

\address{
School of Mathematics, Shandong University, Jinan, 250100, China}

\email{yuchaowang@mail.sdu.edu.cn}

\begin{abstract}
We investigate the density of rational points on Cayley's cubic surface whose coordinates have few prime factors. The key tools used are the circle method and universal torsors.
\end{abstract}


\maketitle

\section{Introduction}
Cayley's cubic surface is defined in $\mathbb{P}^3$ by the equation
\begin{equation*}
S_0:\ x_2x_3x_4+x_1x_3x_4+x_1x_2x_4+x_1x_2x_3=0.
\end{equation*}
It has singularity type $4\mathbf{A}_1$. Moreover, there are 9 lines in the surface, three of which have the form $x_i+x_j=x_k+x_l=0$, and the remaining six have the shape $x_i=x_j=0$. We shall write $U_0$ for the complement of these lines in the surface $S_0$.

Heath-Brown \cite{Hb} considered the density of rational points on Cayley's cubic surface. Let $\mathbb{Z}^4_{\text{prim}}$ be the set of vectors $\mathbf{x}=(x_1,x_2,x_3,x_4)\in\mathbb{Z}^4$ with $\text{gcd}(x_1,x_2,x_3,x_4)=1$. Then we set
\begin{equation*}
N(B)=\#\{\mathbf{x}\in \mathbb{Z}^4_{\text{prim}}:\ [\mathbf{x}]\in U_0, \max |x_i|\leq B \}.
\end{equation*}
Manin (see Batyrev and Manin \cite{BM}) has given a very general conjecture which would predict in this case that there is a suitable positive constant $c$ such that $N(B)\sim cB(\log B)^6$, as $B\rightarrow \infty$. Heath-Brown \cite{Hb} showed that
\begin{equation*}
B(\log B)^6\ll N(B)\ll B(\log B)^6.
\end{equation*}
The approach combines analytic methods with the theory of universal torsors. For details of universal torsors for singular del Pezzo surfaces, we refer the reader to Derenthal \cite{D}.

In this paper we are concerned with the almost prime integral points on Cayley's cubic surface. Let $P_r$ indicate an $r$-almost prime, which is a number with at most $r$ prime factors, counted with multiplicity. For any cubic surface $S$, we write $U$ for the complement of the lines in the surface $S$. Then we define the saturation number $r(S)$ to be the least number $r$ such that the set of $\mathbf{x}\in \mathbb{Z}^4_{\text{prim}}$ for which $[\mathbf{x}]\in U$ and $x_1x_2x_3x_4=P_r$ , is Zariski dense in $U$. The main aim of this paper is to show that $r(S_0)<\infty$.

Bourgain, Gamburd and Sarnak \cite{BGS} and Nevo and Sarnak \cite{NS} established upper bounds for saturation numbers for orbits of congruence subgroups of semi-simple groups acting linearly on affine space. Moreover, Liu and Sarnak \cite{LS} considered the saturation number for certain affine quadric surfaces. Note that Derenthal and Loughran \cite{DL,DL2} have shown that Cayley's cubic surface is not an equivariant compactification of a linear algebraic group. Thus the methods of \cite{BGS} or \cite{NS} do not apply to $S_0$.

To prove that a set of $[\mathbf{x}]\in U_0$ is Zariski dense, it suffices to prove that given $\varepsilon>0$ and any $\boldsymbol{\xi}\in\mathbb{R}^4$ satisfying $[\boldsymbol{\xi}]\in U_0$, there exists $B\in \mathbb{N}$ sufficiently large and at least one point $[\mathbf{x}]$ in the set, such that
\begin{equation*}
\left|\frac{\mathbf{x}}{B}-\boldsymbol{\xi}\right|< \varepsilon.
\end{equation*}
We will prove the following.
\begin{thm}
Let $[\boldsymbol{\xi}]\in U_0$ and define
\begin{equation*}
\begin{split}
M_{U_0}(\boldsymbol{\xi},B,r)=\# \{&\mathbf{x}\in\mathbb{Z}^4_{\text{prim}}:\ [\mathbf{x}]\in U_0,\  \left|\frac{\mathbf{x}}{B}-\boldsymbol{\xi}\right|\ll (\log B)^{-1},\\
&\ x_1x_2x_3 x_4=P_r\}.
\end{split}
\end{equation*}
Then for sufficiently large $B$, we have
\begin{equation*}
M_{U_0}(\boldsymbol{\xi},B,12)\gg B(\log B)^{-7}.
\end{equation*}
The implicit constants are allowed to depend on $\boldsymbol{\xi}$. In particular, the saturation number satisfies $r(S_0)\leq 12$.
\end{thm}
Define $\tilde{r}(S_0)$ to be the least number $\tilde{r}$ such that the set of $\mathbf{x}\in \mathbb{Z}^4_{\text{prim}}$ for which $[\mathbf{x}]\in U_0$ and the product $x_1x_2x_3x_4$ has at most $\tilde{r}$ distinct prime factors, is Zariski dense in $U_0$. Then it is worth pointing out that our methods give $\tilde{r}(S_0)\leq 4$.

The idea of the proof is to apply the theory of universal torsors to specify some integral solutions in a particular form, whose coordinates have few prime factors, and then give a lower bound for the number of such solutions which are close to some fixed real solution via the circle method. More precisely, we use the circle method to count the number of solutions to the equation
\begin{equation*}
\beta_1p_1+\beta_2p_2+\beta_3p_3+\beta_4p_4=0,
\end{equation*}
where $p_i$ are primes which lie in certain intervals and $\beta_i\in\{1,-1\}$, for $i=1,\dots,4$.

It seems likely that similar methods will apply to several other singular cubic surfaces. Let $S_1\subset \mathbb{P}^3$ be the cubic surface given by the equation
\begin{equation*}
x_1x_2x_3=x_4(x_1+x_2+x_3)^2.
\end{equation*}
There is a unique singular point which is of type $D_4$. The density of rational points of bounded height on $S_1$ has been studied by Browning \cite{Br} and Le Boudec \cite{Le}. Arguing similarly as in this paper, we can show that $r(S_1)\leq 12$.
\begin{ack}
\emph{The author wishes to express his sincere appreciation to Tim Browning for introducing him to this problem and giving him various suggestions. This work was carried out while the author was a visiting PhD student at University of Bristol. The author is grateful for the hospitality. While working on this paper the author was supported by the China Scholarship Council.}
\end{ack}
\section{The circle method}
Suppose
\begin{equation*}
F(t_1,t_2,t_3,t_4)=\beta_1t_1+\beta_2t_2+\beta_3t_3+\beta_4t_4
\end{equation*}
is a linear form, with the coefficients $\beta_i\in \{-1,1\}$.
Suppose $\eta_1$, $\dots$, $\eta_4$ are fixed positive real numbers, satisfying
\begin{equation*}
F(\eta_1,\eta_2,\eta_3,\eta_4)=\beta_1\eta_1+\beta_2\eta_2+\beta_3\eta_3+\beta_4\eta_4=0.
\end{equation*}
Write
\begin{equation*}
I_j=[\eta_j B^{\frac{1}{3}}-B^{\frac{1}{3}}(\log B)^{-1},\eta_jB^{\frac{1}{3}}+B^{\frac{1}{3}}(\log B)^{-1}],
\end{equation*}
for $j=1,2,3,4$, where $B$ is a sufficiently large parameter. Furthermore, set
\begin{equation*}
R(B)=\sum_{\stackrel{p_j\in I_j}{F(p_1,p_2,p_3,p_4)=0}}(\log p_1)(\log p_2)(\log p_3)(\log p_4).
\end{equation*}
In this section, we use the circle method to prove the following lemma.
\begin{lem}
For any $A>0$, we have
\begin{equation*}
R(B)=J(B)\mathfrak{S}+O(B(\log B)^{-A}),
\end{equation*}
where $J(B)$ is the number of solutions of
\begin{equation*}
F(m_1,m_2,m_3,m_4)=0
\end{equation*}
with $m_j\in I_j $ and
\begin{equation*}
\mathfrak{S}=\prod_{p}\Big(1+\frac{1}{(p-1)^3}\Big).
\end{equation*}
Moreover, we have $J(B)\gg B(\log B)^{-3}$ and $\mathfrak{S}\gg1$.
\end{lem}
\begin{proof}
Write
\begin{equation*}
L=\log B,\qquad P=L^D, \qquad Q=B^{\frac{1}{3}}P^{-3},
\end{equation*}
where $D$ is a sufficiently large parameter to be chosen later. Furthermore, denote
\begin{equation*}
S_j(\alpha)=\sum_{p_j\in I_j}\log p_j e(\beta_jp_j\alpha ).
\end{equation*}
Then we have
\begin{equation*}
R(B)=\int_{\frac{1}{Q}}^{1+\frac{1}{Q}}S_1(\alpha)S_2(\alpha)S_3(\alpha)S_4(\alpha)d \alpha.
\end{equation*}
By Dirichlet's lemma on rational approximation, each $\alpha\in (\frac{1}{Q},1+\frac{1}{Q}]$ may be written in the form
\begin{equation*}
\alpha=\frac{a}{q}+\lambda,\quad |\lambda|<\frac{1}{qQ},
\end{equation*}
for some integers $a$, $q$ with $1\leq a\leq q\leq Q$ and $(a,q)=1$. Now we define the sets of major and minor arcs as follows:
\begin{equation*}
\mathfrak{M}=\bigcup_{q\leq P}\bigcup_{\stackrel{1\leq a\leq q}{(a,q)=1}}\Big[\frac{a}{q}-\frac{1}{qQ},\frac{a}{q}+\frac{1}{qQ}\Big], \quad \mathfrak{m}=\big(\frac{1}{Q},1+\frac{1}{Q}\big]\setminus \mathfrak{M}.
\end{equation*}
Then
\begin{equation*}
R(B)=\int_{\mathfrak{M}}S_1(\alpha)S_2(\alpha)S_3(\alpha)S_4(\alpha)d \alpha+\int_{\mathfrak{m}}S_1(\alpha)S_2(\alpha)S_3(\alpha)S_4(\alpha)d \alpha.
\end{equation*}
\subsection{The major arcs}
We first estimate the contribution of the integral over the major arcs. For any $\alpha\in\mathfrak{M}$, there exist integers $a$ and $q$ such that
\begin{equation*}
\alpha=\frac{a}{q}+\lambda,\quad 1\leq q \leq P,\quad (a,q)=1 \quad \text{and}\quad |\lambda|<\frac{1}{qQ}.
\end{equation*}
Hence
\begin{equation*}
\begin{split}
S_j(\alpha)&=\sum_{p_j\in I_j}\log p_j e\Big(\frac{\beta_jap_j}{q}\Big)e(\beta_j\lambda p_j)\\
&=\frac{1}{\phi(q)}\sum_{\chi\  \text{mod}\  q}\sum_{\stackrel{h=1}{(h,q)=1}}^{q}e\Big(\frac{\beta_jah}{q}\Big)\bar{\chi}(h)\sum_{p_j\in I_j}\log p_j\chi(p_j)e(\beta_j\lambda p_j).
\end{split}
\end{equation*}
Let
\begin{equation*}
W_j(\chi,\lambda)=\sum_{m_j\in I_j}(\Lambda(m_j)\chi(m_j)-\delta_{\chi})e(\beta_j\lambda m_j),
\end{equation*}
and
\begin{equation*}
\widehat{W}_j(\chi,\lambda)=\sum_{p_j\in I_j}\log p_j \chi(p_j)e(\beta_j\lambda p_j)-\sum_{m_j\in I_j}\delta_{\chi}e(\beta_j\lambda m_j),
\end{equation*}
where $\delta_{\chi}=1$ or $0$ according as $\chi$ is principal or not. Then we have
\begin{equation}
W_j(\chi,\lambda)-\widehat{W}_j(\chi,\lambda)=\sum_{k\geq2}\sum_{p_j^k\in I_j}\log p_j \chi(p_j^k)e(\beta_j\lambda p_j^k)\ll B^{\frac{1}{6}}L.
\end{equation}
Thus
\begin{equation}
\begin{split}
&S_j(\alpha)-\frac{\mu(q)}{\phi(q)}\sum_{m_j\in I_j}e(\beta_j\lambda m_j)\\
=&\frac{1}{\phi(q)}\sum_{\chi\,\text{mod}\,q}\sum_{\stackrel{h=1}{(h,q)=1}}^q e\Big(\frac{\beta_jah}{q}\Big)\bar{\chi}(h)\widehat{W}_j(\chi,\lambda)\\
=&\frac{1}{\phi(q)}\sum_{\chi\,\text{mod}\,q}\sum_{\stackrel{h=1}{(h,q)=1}}^q e\Big(\frac{\beta_jah}{q}\Big)\bar{\chi}(h)(\widehat{W}_j(\chi,\lambda)-W_j(\chi,\lambda))\\
&+\frac{1}{\phi(q)}\sum_{\chi\,\text{mod}\,q}\sum_{\stackrel{h=1}{(h,q)=1}}^q e\Big(\frac{\beta_jah}{q}\Big)\bar{\chi}(h)W_j(\chi,\lambda).
\end{split}
\end{equation}
The main tool here is a short intervals version of Siegel-Walfisz theorem. Using (6) in Perelli and Pintz \cite{P}, we see that
\begin{equation*}
\sum_{m_j\in I_j}\Lambda(m_j)\chi(m_j)=\sum_{m_j\in I_j}\delta_{\chi}+O(B^{\frac{1}{3}}L^{-1-5D}).
\end{equation*}
Therefore, integration by parts gives
\begin{equation*}
\begin{split}
W_j(\chi,\lambda)=&\int_{I_j}e(\beta_j\lambda u)d\big( \sum_{m_j\leq u,m_j\in I_j}(\Lambda(m_j)\chi(m_j)-\delta_{\chi})\big)\\
\ll& \big|\sum_{m_j\in I_j}(\Lambda(m_j)\chi(m_j)-\delta_{\chi})\big|\\
&+\Big|\lambda\int_{I_j}e(\beta_j\lambda u)\big( \sum_{m_j\leq u,m_j\in I_j}(\Lambda(m_j)\chi(m_j)-\delta_{\chi})\big) du\Big|\\
\ll& (1+|\lambda|B^{\frac{1}{3}}L^{-1})B^{\frac{1}{3}}L^{-1-5D}.
\end{split}
\end{equation*}
Thus we obtain
\begin{equation}
W_j(\chi,\lambda)\ll (1+|\lambda|B^{\frac{1}{3}}L^{-1})B^{\frac{1}{3}}L^{-1-5D}\ll B^{\frac{1}{3}}L^{-2-2D}.
\end{equation}
Combining (2.1), (2.2) and (2.3), we get
\begin{equation*}
S_j(\alpha)=\frac{\mu(q)}{\phi(q)}\sum_{m_j\in I_j}e(\beta_j\lambda m_j)+O(B^{\frac{1}{3}}L^{-2D}).
\end{equation*}
We have similar results for $S_2(\alpha)$, $S_3(\alpha)$ and $S_4(\alpha)$. Thus we obtain
\begin{equation}
\begin{split}
&\int_{\mathfrak{M}}S_1(\alpha)S_2(\alpha)S_3(\alpha)S_4(\alpha)d \alpha \\
&-\sum_{q\leq P}\frac{{\mu}^2(q)}{{\phi}^4(q)}\sum_{\stackrel{a=1}{(a,q)=1}}^{q}\int_{-\frac{1}{qQ}}^{\frac{1}{qQ}}\sum_{m_j\in I_j}e(\lambda F(m_1,m_2,m_3,m_4))d\lambda\\
&\ll BL^{-D}.
\end{split}
\end{equation}
Note that
\begin{equation*}
\begin{split}
&\int_{-\frac{1}{2}}^{\frac{1}{2}}\sum_{m_j\in I_j}e(\lambda F(m_1,m_2,m_3,m_4))d\lambda\\
&-\int_{-\frac{1}{qQ}}^{\frac{1}{qQ}}\sum_{m_j\in I_j}e(\lambda F(m_1,m_2,m_3,m_4))d\lambda\ll BL^{-D}.
\end{split}
\end{equation*}
Inserting this into (2.4), we obtain
\begin{equation}
\int_{\mathfrak{M}}S_1(\alpha)S_2(\alpha)S_3(\alpha)S_4(\alpha)d \alpha=J(B)\mathfrak{S}(P)+O(BL^{-D}),
\end{equation}
where
\begin{equation*}
\mathfrak{S}(P)=\sum_{q=1}^{P}\frac{{\mu(q)}^2}{{\phi(q)}^3}.
\end{equation*}
\subsection{The minor arcs}
Now we estimate the contribution of the integral over the minor arcs. For any $\alpha\in\mathfrak{m}$, there exist integers $a$ and $q$ such that
\begin{equation*}
P\leq q\leq Q,\quad (a,q)=1 \quad \text{and} \quad \Big| \alpha-\frac{a}{q}\Big|<\frac{1}{qQ}.
\end{equation*}
Using the non-trivial upper bound for the exponential sums over primes in short intervals, we get
\begin{equation*}
S_1(\alpha)\ll B^{\frac{1}{3}}(\log B)^{-A},
\end{equation*}
provided $D$ is chosen to be sufficiently large. Such a result can be found in several references, for example see Theorem 2 in Zhan \cite{Z}. Also, we have the following mean-value estimate:
\begin{equation*}
\int_{\frac{1}{Q}}^{1+\frac{1}{Q}}|S_j(\alpha)|^2 d\alpha \ll B^{\frac{1}{3}}.
\end{equation*}
By Cauchy's inequality, we obtain
\begin{equation*}
\int_{\frac{1}{Q}}^{1+\frac{1}{Q}}|S_2(\alpha)S_3(\alpha)| d\alpha \ll B^{\frac{1}{3}}.
\end{equation*}
Thus we have
\begin{equation}
\int_{\mathfrak{m}}S_1(\alpha)S_2(\alpha)S_3(\alpha)S_4(\alpha)d \alpha\ll BL^{-A}.
\end{equation}
Combining (2.5) and (2.6), we get
\begin{equation*}
R(B)=J(B)\mathfrak{S}(P)+O(BL^{-A}).
\end{equation*}
\subsection{The singular series}
For the singular series, we have
\begin{equation*}
\Big|\frac{{\mu(q)}^2}{{\phi(q)}^3}\Big|\leq \frac{1}{{\phi(q)}^3}.
\end{equation*}
Note that
\begin{equation*}
\mathfrak{S}=\prod_{p}\Big(1+\frac{1}{(p-1)^3}\Big)=\sum_{q=1}^{\infty}\frac{{\mu(q)}^2}{{\phi(q)}^3}\gg 1.
\end{equation*}
Thus we obtain
\begin{equation*}
\mathfrak{S}(P)-\mathfrak{S}\ll\sum_{q>P}\frac{1}{{\phi(q)}^3}\ll L^{-D}.
\end{equation*}
Since $J(B)\ll B(\log B)^{-3}$, then we have
\begin{equation*}
R(B)=J(B)\mathfrak{S}+O(B(\log B)^{-A}).
\end{equation*}
Now we establish the lower bound for $J(B)$. For $j=1,2,3$, we fix $m_j\in[\eta_j B^{\frac{1}{3}}-\frac{1}{3}B^{\frac{1}{3}}(\log B)^{-1},\eta_jB^{\frac{1}{3}}+\frac{1}{3}B^{\frac{1}{3}}(\log B)^{-1}]$. Recall that
\begin{equation*}
\beta_1\eta_1+\beta_2\eta_2+\beta_3\eta_3+\beta_4\eta_4=0,
\end{equation*}
thus we have
\begin{equation*}
\beta_1 m_1+\beta_2 m_2+\beta_3 m_3\in[-\beta_4\eta_4 B^{\frac{1}{3}}-B^{\frac{1}{3}}(\log B)^{-1},-\beta_4\eta_4B^{\frac{1}{3}}+B^{\frac{1}{3}}(\log B)^{-1}].
\end{equation*}
Since
\begin{equation*}
I_4=[\eta_4 B^{\frac{1}{3}}-B^{\frac{1}{3}}(\log B)^{-1},\eta_4B^{\frac{1}{3}}+B^{\frac{1}{3}}(\log B)^{-1}],
\end{equation*}
then we obtain
\begin{equation*}
J(B)\gg B(\log B)^{-3}.
\end{equation*}
Therefore we prove the lemma.
\end{proof}
\section{The universal torsor}
In this section, we use a passage to the universal torsor for Cayley's cubic surface. Details can be found in Section 2 of Heath-Brown \cite{Hb}. It is shown that the universal torsor for Cayley's cubic surface is an open subvariety in
\begin{equation*}
\mathbb{A}^{13}=\text{Spec}\mathbb{Z}[v_{12},v_{13},v_{14},y_1,y_2,y_3,y_4,z_{12},z_{13},z_{14},z_{23},z_{24},z_{34}],
\end{equation*}
defined by six equations of the form
\begin{equation*}
z_{ik}z_{il}y_j+z_{jk}z_{jl}y_i=z_{ij}v_{ij}
\end{equation*}
and three equations of the form
\begin{equation*}
v_{ij}v_{ik}=z_{il}^2y_jy_k-z_{jk}^2y_iy_l,
\end{equation*}
where
\begin{equation*}
z_{ij}=z_{ji},\quad v_{ij}=v_{ji},\quad and \quad v_{ij}=-v_{kl}.
\end{equation*}
In fact, we are not working with the full universal torsor. We need the following lemma, which is essentially a restatement of Lemma 1 in Heath-Brown \cite{Hb}.
\begin{lem}
Let $[\mathbf{x}]\in U_0$ be a primitive integral solution. Then either $\mathbf{x}$ or $-\mathbf{x}$ takes the form
\begin{equation*}
\begin{split}
&x_1=z_{12}z_{13}z_{14}y_2y_3y_4,\\
&x_2=z_{12}z_{23}z_{24}y_1y_3y_4,\\
&x_3=z_{13}z_{23}z_{34}y_1y_2y_4,\\
&x_4=z_{14}z_{24}z_{34}y_1y_2y_3,
\end{split}
\end{equation*}
with non-zero integer variables $y_i$ and positive integer variables $z_{ij}$ constrained by the conditions
\begin{equation*}
\text{gcd}(y_i,y_j)=1,
\end{equation*}
\begin{equation*}
\text{gcd}(y_i,z_{ij})=1,
\end{equation*}
\begin{equation*}
\text{gcd}(z_{ab},z_{cd})=1,
\end{equation*}
for $\{a,b\}$, $\{c,d\}$ distinct, and satisfying the equation
\begin{equation*}
z_{23}z_{24}z_{34}y_1+z_{13}z_{14}z_{34}y_2+z_{12}z_{14}z_{24}y_3+z_{12}z_{13}z_{23}y_4=0.
\end{equation*}
Moreover, none of $z_{23}z_{24}z_{34}y_1+z_{13}z_{14}z_{34}y_2$, $z_{23}z_{24}z_{34}y_1+z_{12}z_{14}z_{24}y_3$ or $z_{23}z_{24}z_{34}y_1+z_{12}z_{13}z_{23}y_4$ may vanish. Conversely, if $y_i$ and $z_{ij}$ are as above, then the vector $\mathbf{x}$ taking the above form, will be a primitive integral solution lying on $U_0$.
\end{lem}
\section{The proof of the theorem}
Recall that $[\boldsymbol{\xi}]\in U_0$. Thus we get $\xi_1\xi_2\xi_3\xi_4\neq0$, and $\xi_k+\xi_l\neq0$ for $k\neq l$. Now fix
\begin{equation*}
\eta_j=\Big|\frac{\sqrt[3]{\xi_1\xi_2\xi_3\xi_4}}{\xi_j}\Big|,\quad \beta_j=\text{sgn}\Big(\frac{\sqrt[3]{\xi_1\xi_2\xi_3\xi_4}}{\xi_j}\Big),
\end{equation*}
for $j=1,2,3,4$. Then we obtain that $\eta_1$, $\eta_2$, $\eta_3$ and $\eta_4$ are positive numbers, satisfying
\begin{equation*}
\beta_k\eta_k+\beta_l\eta_l\neq0,
\end{equation*}
for $k\neq l$, and
\begin{equation*}
\beta_1\eta_1+\beta_2\eta_2+\beta_3\eta_3+\beta_4\eta_4=0.
\end{equation*}
By Lemma 2, we obtain that for sufficiently large $B$, there exists a suitable positive constant $c$, such that there are at least $cB(\log B)^{-7}$ solutions to the equation
\begin{equation*}
\beta_1p_1+\beta_2p_2+\beta_3p_3+\beta_4p_4=0,
\end{equation*}
with $p_j\in I_j$. Among these solutions, there are at most $O(B^{\frac{2}{3}}(\log B)^{-2})$ ones with $\mu(p_1p_2p_3p_4)=0$. Thus we still have at least $c'B(\log B)^{-7}$ solutions satisfying $\mu(p_1p_2p_3p_4)\neq0$, where $c'$ is a positive constant. In Lemma 3, we fix $z_{12}=z_{13}=z_{14}=z_{23}=z_{24}=z_{34}=1$. Then the relations become
\begin{equation*}
y_1+y_2+y_3+y_4=0,
\end{equation*}
\begin{equation*}
\text{gcd}(y_i,y_j)=1,
\end{equation*}
and none if $y_1+y_2$, $y_1+y_3$ or $y_1+y_4$ may vanish. We also have
\begin{equation*}
\begin{split}
&x_1=y_2y_3y_4,\\
&x_2=y_1y_3y_4,\\
&x_3=y_1y_2y_4,\\
&x_4=y_1y_2y_3.
\end{split}
\end{equation*}
Now we set $y_j= \beta_jp_j$. Thus we obtain
\begin{equation*}
x_i=\prod\limits_{\stackrel{j=1}{j\neq i}}^{4}\beta_jp_j.
\end{equation*}
Obviously, we get $\mathbf{x}\in \mathbb{Z}^4_{\text{prim}}$, $[\mathbf{x}]\in U_0$ and $\mathbf{x}$ primitive. For $x_1$, we have
\begin{equation*}
\left|\frac{x_1}{B}-\beta_2\beta_3\beta_4\eta_2\eta_3\eta_4\right|\ll (\log B)^{-1}.
\end{equation*}
Note that $\beta_2\beta_3\beta_4\eta_2\eta_3\eta_4=\xi_1$, then
\begin{equation*}
\left|\frac{x_1}{B}-\xi_1\right|\ll (\log B)^{-1}.
\end{equation*}
Arguing similarly as above, we get
\begin{equation*}
\left|\frac{\mathbf{x}}{B}-\boldsymbol{\xi}\right|\ll (\log B)^{-1}.
\end{equation*}
Thus the proof is concluded.


\end{document}